\documentclass[11pt,reqno]{amsart}
\usepackage{color,mathrsfs,epsfig}
\usepackage{psfrag,graphicx}
\usepackage{amssymb}

\newcommand{\N}{\mathbb{N}}
\newcommand{\reali}{\mathbb{R}}
\newcommand{\rpic}{{\mathbb{R}^+}}

\newcommand{\pt}{\partial}

\renewcommand{\div}{{\rm div}\,}

\newcommand{\Mes}{\mathcal{M}}
\newcommand{\M}{\mathbf{M}}

\renewcommand{\det}{{\rm det}}

\newcommand{\lip}{{\rm Lip}}

\newcommand{\dst}{\displaystyle}

\newcommand{\C}[1]{\mathscr{C}^{#1}}
\newcommand{\Cc}[1]{\mathscr{C}_c^{#1}}
\newcommand{\modulo}[1]{{\left|#1\right|}}
\newcommand{\norma}[1]{{\left\|#1\right\|}}
\renewcommand{\L}[1]{{\mathbf{L}^#1}}

\newcommand{\W}[2]{{\mathbf{W}^{#1,#2}}}

\newcommand{\Lloc}[1]{{\mathbf{L}_{loc}^{#1}}}

\renewcommand{\d}[1]{\mathinner{\mathrm{d}{#1}}}
\newcommand{\conv}{{\boldsymbol {*}}}
\renewcommand{\leq}{\leqslant}
\renewcommand{\geq}{\geqslant}

\swapnumbers
\newtheorem{theorem}[subsection]{Theorem}
\newtheorem{lemma}[subsection]{Lemma}
\newtheorem{proposition}[subsection]{Proposition}
\newtheorem{corollary}[subsection]{Corollary}

\theoremstyle{definition}
\newtheorem{definition}[subsection]{Definition}

\theoremstyle{remark}
\newtheorem{remark}[subsection]{Remark}

\numberwithin{equation}{section}

\title[System of continuity equations with non-local flow]{Existence and uniqueness of measure solutions for a system of continuity equations with non-local flow}
\author{Gianluca Crippa}
\address{Departement Mathematik und Informatik, 
Universitaet Basel, 
Rheinsprung 21, CH-4051 Basel, Switzerland,
\texttt{gianluca.crippa@unibas.ch}}

\author{Magali L\'ecureux-Mercier}
\address{
Technion -- Israel Institute of Technology, Amado Building, 
Haifa 32000,
Israel,
\texttt{mercier@tx.technion.ac.il}}

\date{\today}
\begin{document}
\bibliographystyle{plain}
\begin{abstract}
In this paper, we prove existence and uniqueness of measure solutions for the Cauchy problem associated to the (vectorial) continuity equation with a non-local flow. We also give a stability result with respect to various parameters. 
%A similar equation, in the scalar was previously studied in \cite{ColomboHertyMercier} in the framework of Kru\v zkov weak entropy solution: we recover their results under weaker assumptions.

 \medskip


  \medskip


\end{abstract}

\maketitle

%
% {\small
% To do :  
% \begin{itemize}
%\item dependence with respect to time.
%\end{itemize}
%
% Possible extensions:
% \begin{itemize}
%\item coincide with entropy solutions?
% \item dependence of $V^i$ with respect to $t$,
% %\item positivity of $\bar \rho^i$ ?
% \item differentiability with respect to $\bar \rho$ ?
% \item existence in the case $V^i$ Sobolev ?
% \item asymptotic behaviour ?
% \item controlability ?
% \item ....
% \end{itemize}
%}

%%%%%%%%%%%%%%%%%%%%%%%%%%%%%%%%%%%%%%%%%

%%%%%%%%%%%%%%%%%%%%%%%%%%%%%%%%%%%%%%%%%
\section{Introduction}

In this paper, we consider the system of nonlocal continuity equations
\begin{equation}\label{eq:system}
\left\{\begin{array}{ll}
 \pt_t \rho^i +\div \left(\rho^i V^i(t, x, \eta^i \conv \rho)\right)=0\,,& \qquad t\in \rpic\,,\; x\in \reali^d\,,\\[3pt]
\rho^i(0)=\bar \rho^i\,,&\qquad i \in \{1,\ldots, k\}\,,
\end{array}
\right.
\end{equation}
where the unknown $\rho=(\rho^1, \ldots, \rho^k)$ is a vector of measures,  $\eta^i=(\eta^{i,1}, \ldots , \eta^{i ,k})$ is a vector of convolution kernels and we set  $\eta^i \conv \rho= (\eta^{i,1}*\rho^1, \ldots,\eta^{i,k}*\rho^k)$. For any time $t\geq 0$, if $\mu_t\in \Mes^+(\reali^d)$ is a bounded measure on $\reali^d$ and $\eta_t$ is a bounded function on $\reali^d$, then the convolution is taken with respect to space only and is defined as usually as $\mu_t*\eta_t =\int_{\reali^d}\eta_t(x-y)\d{\mu_t(y)}$.

For example in \cite{ColomboHertyMercier}, the authors consider the scalar conservation law
where $V$ is a nonlocal functional with respect to $\rho$. This equation stands for various models such as a sedimentation model, a supply-chain model, a pedestrian traffic model. For physical reason, in the following we are looking for positive solutions possibly with concentration, i.e.  for any time $t$  the solution has to be in $\Mes^+(\reali^d)^k$.
 
Our goal here is to improve the results of \cite{ColomboHertyMercier}, not only by considering a system, but also 
lightening the hypotheses on $V$ and $\eta$.
We prove here existence and uniqueness of weak measure solutions to (\ref{eq:system}).
Let us introduce the following sets of hypotheses:
\begin{description}
\item[(V)] The vector field $V (t,x,r) : \rpic\times \reali^d\times \reali^k\to \M_{d,k}$ is uniformly bounded and it is  Lipschitz in $(x,r) \in \reali^d\times \reali^k$ uniformly in time: 
$$
V\in \L\infty(\rpic\times\reali^d\times \reali^k)\cap\L\infty(\rpic, \lip(\reali^d\times\reali^k,
\M_{d,k}))\,.
$$
\item[($\boldsymbol{\eta}$)] The convolution kernel $\eta (t,x) : \rpic \times \reali^d\to \M_k$  is uniformly bounded and it is Lipschitz in $x\in \reali^d$ uniformly in time: 
$$
\eta\in \L\infty(\rpic\times\reali^d)\cap \L\infty(\rpic, \lip(\reali^d, \M_{k}))\,.
$$
\end{description}
The main result of this paper is the following theorem.
\begin{theorem}\label{thm:main}
Let $\bar\rho \in \Mes^+(\reali^d)^k$. Let us assume that $V$ satisfies \textbf{(V)} and $\eta $ satisfies \textbf{($\boldsymbol{\eta}$)}. Then there exists a unique solution $\rho\in \L\infty(\reali^+, \Mes^+(\reali^d)^k)$ to (\ref{eq:system}) with initial condition $\bar \rho$.
\end{theorem}
We refer to Section \ref{sec:solution} for precise notations and definitions, in particular for the notion of solution. 

\begin{remark}\label{rk:propsol} Assume $V$ satisfies \textbf{(V)} and $\eta $ satisfies \textbf{($\boldsymbol{\eta}$)}. Then Theorem \ref{thm:main} is completed by the following properties: 
\begin{itemize}
\item If $\bar\rho\in \L1(\reali^d, (\reali^+)^k)$ then $\rho\in \C0(\rpic,\L1(\reali^d, (\reali^+)^k))$, up to redefinition on a negligible set of times, and for all time $t\geq 0$, for all $i\in \{1,\ldots, k\}$ we have $\norma{\rho^i(t)}_{\L1}=\norma{\bar\rho^i}_{\L1}$.
\item If $\bar\rho\in (\L1\cap\L\infty)(\reali^d, (\reali^+)^k)$ then $\rho\in \Lloc\infty(\rpic, \L\infty(\reali^d, (\reali^+)^k ))$ and for all time $t\geq 0$, we have $\norma{\rho(t)}_{\L\infty}\leq \norma{\bar \rho}_{\L\infty} e^{Ct}$
with $C$ a constant dependent on $\lip_x(V)$, $\lip_r(V)$, $\lip_x(\eta)$ and  $\norma{\bar \rho}_{\Mes}$.
\item Let $\bar \rho, \bar \sigma \in \Mes^+(\reali^d)^k$ such that for any $i$, $\norma{\bar\rho^i}_{\Mes}=\norma{\bar \sigma^i}_{\Mes}$. Let $\rho$ and $\sigma$ be the solutions of (\ref{eq:system}) associated to the initial conditions $\bar \rho$ and $\bar \sigma$,  then we have the estimate:
\[
 \mathscr{W}_1(\rho_t, \sigma_t)\leq e^{Kt}\mathscr{W}_1(\bar\rho, \bar \sigma)\,,
\]
where $K=\lip_x(V)+\lip_r (V)\lip_x(\eta) \norma{\bar \rho}_{\Mes}+ \lip_r(V)\lip_x (\eta) \norma{\bar\rho}_{\Mes}$ and  $\mathscr{W}_1(\rho_t, \sigma_t)$ is the Wasserstein distance of order one between $\rho_t$ and $\sigma_t$.
\end{itemize}
These properties are described in Corollary \ref{cor:L1} and in Proposition \ref{prop:stabGeneral}. The Wasserstein distance of order one is rigorously defined in Section \ref{sec:tools}.
\end{remark}

\begin{remark}
In Theorem \ref{thm:main} as well as in the other results of this papers, it is in fact sufficient to require that $V^i(t,x,r)$ is $\L\infty$ in $t, x$ and $\Lloc\infty$ in $r$. Indeed, $\rho\conv \eta^i$ is uniformly bounded by $\norma{\bar\rho}_{\Mes}\norma{\eta}_{\L\infty}=M$. Consequently, denoting $B_M$ the  closed ball of center 0 and radius $M$ in $\reali^k$,  it is sufficient to have $V^i\in \L\infty(\rpic\times\reali^d\times B_M)$. 

Note also that, restricting the definition of $V$ and $\eta$ to the time interval $[0,T_0]$, we obtain a solution defined on the same time interval. Consequently, we can as well ask only $V$ and $\eta$ to be $\Lloc\infty$ in time instead of $\L\infty$.
\end{remark}

The system (\ref{eq:system}) stands for a variety of models. Let us present  first a macroscopic model of pedestrian traffic.
In a macroscopic pedestrian crowd model, $\rho$ is the density of the crowd at time $t$ and position $x$ and $V$ is a vector field giving the speed of the pedestrian. 
According to the choice of $V$,  various behaviors can be observed. Several authors already studied pedestrian traffic in  two dimensions space ($N=2$). Some of these models are local in $\rho$ (see \cite{BellomoDogbe_review, CosciaCanavesio, Hughes1, Hughes2, MauryChupin, MauryChupinVenel}) ; other models use not only the local density $\rho(t,x)$ but the entire distribution of $\rho$, typically they depend on the convolution product $\rho(t)*\eta$ (see \cite{ColomboGaravelloMercier, ColomboHertyMercier, ColomboLecureux, ColomboMercier, DiFrancesco, PiccoliTosin}) which represents the spatial average of the density. 
Within the framework of (\ref{eq:system}), we can study the models presented in 
 \cite{ColomboHertyMercier, ColomboLecureux, ColomboMercier}.
 In \cite{ColomboHertyMercier}, the authors considered for $V$ the expression
\[
V= v(\rho*\eta)\vec v(x)\,,
\]
where $v$ is a scalar function giving the speed of the pedestrians; $\eta$ is a convolution kernel averaging the density; and  $\vec v(x)$ is a bounded vector field giving the direction the pedestrian located in $x$ will follow. This model is more adapted to the case of panic in which pedestrians will not deviate from their trajectory and will adapt their velocity to the averaged density. Indeed, even if the density is maximal on a given trajectory, if the averaged  density is not maximal, the pedestrians will push, trying all the same to reach their goal. This behavior can be associated with rush phenomena in which people can even die due to overcompression  (e.g. on Jamarat Bridge in Saudi Arabia, see \cite{HelbingJohanssonZein}).
A similar model was introduced in Piccoli \& Tosin \cite{CristianiPiccoliTosin, PiccoliTosin}, where the authors instead of an isotropic convolution kernel, consider a nonlocal functional taking into account the direction in which the pedestrians are looking. 

In \cite{ColomboHertyMercier},  the authors study the scalar case in the framework of Kru\v zkov entropy solutions. They obtained existence and uniqueness of weak entropy solutions under the hypotheses $v\in \W2\infty(\reali^+, \reali^+)$, $\vec v\in (\W2\infty\cap\W21)( \reali^d, \reali^d)$, and $\eta\in (\W2\infty\cap\L1)(\reali^d, \reali)$.   
%A natural extension to the scalar model above consists in considering several populations of respective densities $\rho^1, \ldots, \rho^k$, with respective speeds $V^1, \ldots, V^k$. This system was studied in \cite{ColomboLecureux}, always through Kru\v zkov entropy solutions; so the authors obtain once again existence and uniqueness of weak entropy solution under very strong hypotheses.
This result was slightly improved in \cite{ColomboLecureux} where, under the same set of hypotheses on $v$, $\vec v$ and $\eta$, the authors consider a system instead of a scalar equation and obtain global in time existence and uniqueness of entropy solutions.
We recover these results with lighter hypotheses. Indeed, although we consider weak measure solutions, these in fact  are unique and consequently coincide  with  entropy solutions when the initial condition is in $\L1$.

%The model (\ref{eq:system}) can correspond to a pedestrian flow model with multi-population. 
%In the case $k=d=2$, we can modelize the interaction between two populations going in opposite directions. To do that,  with $v$ such that $v(x)=1-x$ for $x\in [0,1]$ and $v(x)=0$ for $x\geq 1$, we can define the velocity fields $V^1$ and $V^2$ as  
%\begin{align*}
%V^1(x,r_1,r_2)&=v\left(\frac{r_1+r_2}{2}\right)e_1\,,& V^2(x,r_1, r_2)=-v\left(\frac{r_1+r_2}{2}\right)e_1\,,
%\end{align*}
%where $(e_1, e_2)$ is an orthonormal base in $\reali^2$.

Another model of crowd dynamics that  we recover consists in the coupling of a group of density $\rho(t,x)$ with an isolated agent located in $p(t)$. This  can modelize for example the interaction between groups of preys of densities $\rho$ and an isolated predator located in $p$. Such a model was introduced in \cite{ColomboMercier} where the authors obtained existence and uniqueness of weak entropy solutions under  very strong hypotheses. %There, the system considered had a more general nonlinerarity in $\rho$, which allows to obtain a uniform $\L\infty$ bound.
We  recover here  partially the results concerning the coupling PDE/ODE of \cite{ColomboMercier}. %, considering a dirac solution for the second equation. 
Indeed,  the measure framework allows us also to introduce particles/individuals through Dirac measures. For instance, let us assume that $k=k_0+k_1$ such that $\rho^1, \ldots, \rho^{k_0}$ are in fact functions belonging to  $\L\infty(\rpic, \L1(\reali^d, \rpic))$ and that 
 $\delta_{p^1},\ldots, \delta_{p^{k_1}}$ are Dirac measures located in $p^1(t), \ldots, p^{k_1}(t)\in \reali^d$. Let us denote
 $\rho=(\rho^1, \ldots, \rho^{k_0})\in \reali^{k_0}$ and $p=(p^1, \ldots, p^{k_1})\in \M_{d,k_1}$. We also denote with $V^i$ (resp.  $\eta^i$) the vector fields (resp. kernels) associated to $\rho$, and with $U^i$ (resp. $\lambda^i$) the vector fields  (resp. kernels)  associated to $p$.
Note that $\delta_{p^j}*\lambda^{i,j}(x)=\lambda^{i,j} (x-p^j)$. 
By definition of weak measure solution (see Definition \ref{def:weaksol}), 
if $p_i\in \C1([0,T], \reali^d)$, the Dirac measures are satisfying, 
for any $i\in \{1, \ldots, k_1\}$
\begin{eqnarray*}
 \dot p_i(t)=U^i\left[t,\, p_i(t),\, \rho_t\conv \eta_t^i(p_i(t)), \,\lambda^{i,1}\big(p_i(t)-p_1(t)\big),\ldots,\lambda^{i,k_1}\big(p_i(t)-p_{k_1}(t)\big) \right] \,,
\end{eqnarray*}
which can be rewritten
\begin{eqnarray*}
\dot p^i(t)= \Phi^i\left(t, p(t), \rho_t\conv\eta_t^i\left(p^i(t) \right) \right)\,. 
\end{eqnarray*}
Consequently, in this case, the system (\ref{eq:system}) becomes
\[
\left\{
\begin{array}{ll}
 \pt_t \rho^i +\div \left(\rho^i \,V^i\left( t, \,x,\, \eta_t^i \conv \rho_t, \,\left(\lambda^{i,j} (x-p^j(t))\right)_{j=1}^{k_1}  \right)\right)=0\,,\qquad &i \in \{1,\ldots, k_0\}\,,\\[5pt]
 \dot p^j(t)=\Phi^j\left(t, \,p(t),\, \rho_t\conv\eta_t^j (p^j(t)) \right)\,, &j \in \{1,\ldots, k_1\}\,.\\
\end{array}
\right.
\]
So we are coupling ODE with conservation laws.

System (\ref{eq:system}) can also stands for models of aggregation, studied in \cite{Carrillo} under weaker hypotheses admitting   singular kernels. 

The system (\ref{eq:system}) comprised also a model of particles' sedimentation 
$$\pt_t \rho+\pt_x ((\rho*\eta) \, \rho )=0$$ 
which has been introduced \cite{Rubinstein}  and  studied in \cite{Zumbrun}, where the author proved existence and uniquness of weak solutions with initial condition in $\L\infty$. 

Finally, a similar nonlocal model is the one the supply-chain model \cite{arm, arm2}, in which we consider the nonlocal term $\int_0^1 \rho(t,x)\d{x}$ instead of a convolution product. This last model was studied for example in  \cite{CoronKawskiWang} with furthermore boundary conditions in $x=0$ and $x=1$.

\medskip

The proof of Theorem  \ref{thm:main} is divided into two main steps. First, we prove some a priori properties of the solutions (see Section \ref{sec:solution}): mainly, we prove that the weak measure solutions of (\ref{eq:system}) coincide with the Lagrangian solutions of this system. Important consequences are the conservation of the regularity of the initial condition  and the strong continuity in time in the case the solution is a function, as stated in Remark \ref{rk:propsol}. 

Second, we prove the existence and uniqueness of Lagrangian solutions thanks to a fixed point argument (see Section \ref{sec:proofmaintheorem}). Indeed, introducing the set of probability measures endowed with the Wasserstein distance of order one,  we are able to prove a stability estimate with respect to the nonlocal term (see Section \ref{sec:stability}). The technique used  there  is quite similar to the one of Loeper \cite{Loeper}, who studied the Vlasov-Poisson equation and the Euler equation in vorticity formulation.

This article is organized as follows: in Section \ref{sec:solution} we define the two different notions of solution and prove that they coincide. In Section \ref{sec:tools} we give some useful tools on optimal tranport; in Section \ref{sec:stability} we prove an important lemma giving a stability estimate and in Section \ref{sec:proofmaintheorem}, we give the proof of Theorem \ref{thm:main}.

%%%%%%%%%%%%%%%%%%%%%%%%%%%%%%%%%%%%%%%%%%%

%%%%%%%%%%%%%%%%%%%%%%%%%%%%%%%%%%%%%%%%%%%
\section{Notion of solutions}\label{sec:solution}

%%%%%%%%%%%%%%%%%%%%%%%%%%%%%%%%%%%%%%%%%%%
\subsection{General notations}
Let $d\in \N$ be the space dimension and $k\in \N$ be the size of the system. 
In the following, $\M_{d,k}$ is the set of  matrices of size $d\times k$ with real values and $\M_{k}$ is the set of  matrices of size $k\times k$ with real values.  We denote by $\Mes(\reali^d)$ (resp. $\Mes^+(\reali^d)$) the set of bounded (resp. bounded and positive) measure on $\reali^d$ and by $\mathcal{P}(\reali^d)$ the set of probability measures on $\reali^d$, that is the set of bounded positive measures with total mass 1. 

In the following the Lipschitz norms with respect to $x$ or $r$ are taken uniformly with respect to the other variables.  That is to say, for example:
\[
\lip_x(V)=\sup_{t\in \rpic, r\in \reali^k} \{\lip_x(V(t, \cdot, r)) \}\,.
\]

Let us also underline that in the computations, we considered  the norm 1 on the vectors in $\reali^k$. When considering another norm, a constant depending on $k$ appears in the various estimates. Similarly, if $\rho=(\rho^1, \ldots, \rho^k)\in \Mes(\reali^d)^k$ we define the total measure of $\rho$ as $\norma{\rho}_{\Mes}=\norma{\rho^1}_{\Mes}+\ldots +\norma{\rho^k}_{\Mes}$.

The space $\L\infty([0,T], \Mes^+(\reali^d))$ consists of the parametrized measures $\mu=(\mu_t)_{t\in [0,T]}$
 such that, for any $\phi \in \Cc0 (\reali^d, \reali)$,  the application $t\mapsto \int_{\reali^d} \phi \d{\mu_t(x)}$ is measurable and such that $\textrm{ess}\sup_{t\in [0,T]} \norma{\mu_t}_{\Mes}<\infty$.
 
\subsection{Weak measure solutions}\label{def:weaksol}
We say that $\rho \in \L\infty([0,T],\Mes^+(\reali^d)^k)$ is a \emph{weak measure solution} of (\ref{eq:system}) with initial condition $\bar \rho\in \Mes^+(\reali^d)^k$ if, for any $i\in \{1, \ldots, k\}$ and for any test-function $\phi \in \Cc\infty(]-\infty, T[\times \reali^d, \reali)$ we have
\[
\int_0^T \int_{\reali^d} \big[ \pt_t\phi + V^i(t, x, \rho\conv \eta^i)\cdot\nabla\phi \big] \d{\rho^i_t (x)}\d{t} +\int_{\reali^d} \phi(0,x)\d{\bar \rho^i(x)}=0\,.
\]

\begin{remark}
A priori for weak measure solutions of  the continuity equation $\pt_t \rho+\div (\rho b)=0$, with a given vector field $b$, we have only continuity in time for the weak topology (see \cite{DafermosBook}), that is to say, for all $i\in \{1, \ldots, k\}$, for all $\phi\in \Cc0(\reali^d, \reali)$, the application $t\mapsto \int_{\reali^d} \phi(x)\d{\rho^i_t(x)}$  is continuous, up to redefinition of $\rho_t$ on a negligible set of times.  

In the case of the system (\ref{eq:system}),  we have a gain of regularity in time when the initial condition is a function in
$ \L1(\reali^d, (\reali^+)^k)$ (see Corollary \ref{cor:L1}).
\end{remark}

%%%%%%%%%%%%%%%%%%%%%%%%%%%%%%%%%%%%%%%%%%%
\subsection{Push-forward and change of variable}
When $\mu $ is a measure on $\Omega$ and $T:\Omega\to \Omega'$ a measurable map, we denote $T_\sharp \mu$ the push-forward of $\mu$, that is  the measure on $\Omega'$ such that, for every $\phi\in \Cc0(\Omega', \reali)$, 
$$\int_{\Omega'}\phi(x)\d{T_\sharp\mu(x)}=\int_{\Omega}\phi\left( T(y)\right) \d{\mu(y)}\,.$$  

If we assume that $\mu$ and $\nu= T_\sharp \mu$ are absolutely continuous with respect to the  Lebesgue measure so that there exist $f, g\in \L1$ such that $\d\mu(x)=f(x)\d{x}$ and $\d\nu(y)=g(y)\d{y}$, and that $T$ is a $\lip$-diffeomorphism, then we have the \emph{change of variable formula}
\begin{equation}\label{eq:changeofvariable}
f(x)=g(T(x))\modulo{\det(\nabla T(x))}\,.
\end{equation}

Besides, we denote $\mathbb{P}_x:\reali^d\times\reali^d\to \reali^d$ the \emph{projection on the first coordinate}; that is, for any $(u,v)\in \reali^d\times\reali^d$,  $\mathbb{P}_x(u,v)=u$. In a similar way, $\mathbb{P}_y:\reali^d\times\reali^d\to \reali^d$ is the \emph{projection on the second coordinate}; that is, for any $(u,v)\in \reali^d\times\reali^d$,  $\mathbb{P}_y(u,v)=v$.

\subsection{Lagrangian solutions}\label{def:sollag}
We say that $\rho\in \L\infty( [0,T], \Mes^+(\reali^d)^k)$ is a Lagrangian solution  of (\ref{eq:system}) with initial condition $\bar \rho\in \Mes^+(\reali^d)^k$ if, 
 for any $i\in \{1, \ldots, k\}$, there exists an ODE flow $X^i: [0,T]\times \reali^d \to \reali^d$, that is a solution of 
\[
\left\{
\begin{array}{rl}
\dst\frac{\d{ X^i}}{\d{t}}(t,x)=&V^i\left(t, X^i(t,x), \rho_t\conv \eta_t^i(X^i(t,x)) \right)\,,\\[7pt]
\dst X^i(0,x)=&x\,;
\end{array}
\right.
\]
and such that $\rho^i_t={X^i_t}_\sharp \bar\rho^i$ where $X^i_t:\reali^d\to\reali^d$ is the map defined as $X^i_t(x)=X^i(t,x)$ for any $(t,x)\in [0,T]\times\reali^d$.

\begin{remark}\label{rk:blip}
Assume $V$ satisfies \textbf{(V)} and  $\eta$ satisfies \textbf{($\boldsymbol{\eta}$)}. Then, for any $\rho\in \L\infty([0,T], \Mes^+(\reali^d)^k)$, the vector fields $b=V(t, x,\rho_t\conv \eta_t)$  Lipschitz in $x$ and 
\[
\lip_x(b)\leq \lip_x (V)+\lip_r(V) \lip_x(\eta)\norma{\rho_t}_{\Mes}\,.
\] 
Consequently, if $\norma{\rho_t}_{\Mes}$ is uniformly bounded, the ODE flow $X^i$ above is always well-defined, for a fixed $\rho$.

If $\bar \rho \in\L1(\reali^d, \reali^+)$, then the push-forward  formula (\ref{eq:changeofvariable}) becomes, for a.e. $(t,x)\in\rpic\times \reali^d$, 
\begin{equation}\label{eq:cov}
\rho^i(t,X^i(t,x))=\bar\rho^i (x)\exp\left(-\int_{0}^t \div V^i\left(\tau, X^i(\tau,x),\rho_\tau*\eta_\tau(X^i(\tau, x)) \right) \d{\tau}\right) \,.
\end{equation}
\end{remark}

We now show that the two notions of solution in fact coincide.

%%%%%%%%%%%%%%%%%%%%%%%%%%%%%%%%%%%%%
\begin{theorem}
If $\rho$ is a Lagrangian solution of (\ref{eq:system}), then $\rho$ is also a weak measure solution of (\ref{eq:system}).
Conversely, if $\rho$ is a weak measure solution of (\ref{eq:system}), then $\rho$ is also a Lagrangian solution of (\ref{eq:system}).
\end{theorem}

\begin{proof}
\noindent\textbf{1.}  Let $\rho$ be a Lagrangian solution of (\ref{eq:system}).  Let us denote $b^i=V^i(t, x, \rho\conv\eta^i)$ and let  $X^i$ be the ODE flow associated to $b^i$. Then, for any $\phi\in \Cc\infty(]-\infty, T[\times \reali^d, \reali)$,  we have
\begin{eqnarray*}
&&\int_0^T\int_{\reali^d} \left(\pt_t \phi (t,x)+b^i(t,x)\cdot \nabla\phi (t,x)\right)\d{\rho_t(x)}\d{t}\\
&=& \int_0^T \int_{\reali^d}\left( \pt_t\phi(t,X_t^i(x))+b^i(t,X^i_t(x) )\cdot\nabla\phi (t,X_t^i(x))\right)\d{\bar\rho(x)}\d{t}\\
&=&\int_0^T \int_{\reali^d} \frac{\d{}}{\d{t}}\left(\phi(t,X_t^i(x)) \right) \d{\bar\rho(x)}\d{t}\\
&=&\int_{\reali^d} \phi(T,X^i(T,x))\d{\bar\rho(x)}-\int_{\reali^d}\phi(0,x)\d{\bar\rho(x)} = -\int_{\reali^d}\phi(0,x)\d{\bar\rho(x)}\,,
\end{eqnarray*}
which proves that $\rho$ is also a weak measure solution.

\noindent\textbf{2.}
Let $\rho$ be a weak measure solution of (\ref{eq:system}). For any $i\in \{1, \ldots, k\}$,  let us denote $b^i(t,x)=V^i(t, x, \rho\conv\eta^i)$. Let $\sigma$ be the Lagrangian solution of the equation
\begin{equation}\label{eq:lag}
\pt_t\sigma^i +\div(\sigma^i b^i)=0\,,\qquad\qquad\sigma^i(0,\cdot)=\bar\rho^i\,,
\end{equation}
which exists and is unique since $b^i$ is Lipschitz as noted in Remark \ref{rk:blip}.
Then, arguing similarly as in point 1,  $\sigma$ is also a weak measure solution to (\ref{eq:lag}). 
Denoting $u^i=\rho^i-\sigma^i$, we obtain that $u^i$ is a weak measure solution of the equation $\pt_t u^i+\div (u^i\,b^i)=0$ with initial condition $u^i(0,\cdot)=0$. 
Consequently, for any $\phi\in \Cc\infty(]-\infty, T[\times \reali^d, \reali)$,
\[
\int_0^T\int_{\reali^d} \left(\pt_t \phi+b^i(t,x) \cdot\nabla\phi\right)\d{u_t}\d{t}=0\,.
\]
Let $\psi \in \Cc0(]-\infty, T[\times\reali^d, \reali)$. Since $b^i\in \L\infty([0,T]\times \reali^d, \reali^d)$ is Lipschitz in $x$, by computation along the characteristics, we can find $\phi\in \Cc1(]-\infty, T[\times\reali^d, \reali )$ so that $\psi=\pt_t\phi+b^i(t,x)\cdot \nabla \phi$. Hence,   for any $\psi\in \Cc0(]-\infty, T[\times\reali^d, \reali)$, we have $\int_0^T \int_{\reali^d}\psi \d{u_t}\d{t}=0$, which implies $u\equiv 0$ a.e. and $\rho=\sigma$ a.e.
Consequently, we have also $b^i(t,x)=V^i(t, x, \sigma\conv\eta^i)$, and $\sigma=\rho$ is  finally a Lagrangian solution of (\ref{eq:system}). 
\end{proof}

\begin{definition}\label{def:weaklag}
As a consequence of the previous theorem, in the following we simply call \emph{solution} of (\ref{eq:system}) a weak measure solution or a Lagrangian solution of (\ref{eq:system}), that in fact coincide.
\end{definition}

It is now possible to prove  some of the properties given in Remark \ref{rk:propsol}.
\begin{corollary}\label{cor:L1}
 Assume that $V$ satisfies \textbf{(V)} and $\eta $ satisfies \textbf{($\boldsymbol{\eta}$)}. 
 Let $\rho\in \L\infty([0,T], \Mes^+(\reali^d)^k)$ be a solution to (\ref{eq:system}) with initial condition $\bar\rho\in \Mes^+(\reali^d)^k$.
\begin{itemize} 
\item If $\bar\rho\in  \L1(\reali^d,(\reali^+)^k)$. Then we have $\rho\in \C0([0,T], \L1(\reali^d, (\reali^+)^k))$ and for all time $t\in [0,T]$, all $i\in\{1, \ldots, k\}$,  $\norma{\rho^i(t)}_{\L1}=\norma{\bar\rho^i}_{\L1}$.
\item 
If furthermore $\bar \rho \in (\L1\cap\L\infty)(\reali^d,(\reali^+)^k)$, then for all $t\in [0,T]$ we have $\rho(t)\in \L\infty(\reali^d, (\reali^+)^k)$  and we have the estimate 
\[
\norma{\rho(t)}_{\L\infty}\leq \norma{\bar\rho}_{\L\infty}e^{Ct}\,,
\]
where $C$ depends on $\norma{\bar\rho}_{\Mes}$,  $V$ and $\eta$.
\end{itemize}
\end{corollary}

\begin{proof}
Let $\rho$ be a solution of (\ref{eq:system}) with initial condition $\bar \rho\in \L1(\reali^d, (\rpic)^k)$. According to Definition \ref{def:weaklag}, $\rho$ is a Lagrangian solution associated to a flow $X$ and we have immediately that $\norma{\bar\rho}_{\L1}=\norma{\rho(t)}_{\L1}$. 

Besides, as $b^i(t,x)=V^i(t, x,\rho\conv\eta^i)\in \L\infty([0,T]\times \reali^d, \reali^d)$  is bounded in $t$ and Lipschitz in $x$, then $X_t^i\in \lip(\reali^d, \reali^d)$ and we can use the change of variable formula (\ref{eq:cov}). If $\bar \rho\in \L\infty(\reali^d, \reali^k)$, with 
\[
C=\lip_x(V)+\lip_r(V)\lip_x(\eta) \norma{\bar \rho}_{\Mes}\,,
\]
we obtain the desired $\L\infty$ bound and $\rho(t)\in \L\infty$ for all $t\in [0,T]$.

The continuity in time can be proved  directly  by estimating $\norma{\rho_t-\rho_s}_{\L1}$ using Egorov Theorem. This computation is straightforward although a bit long so we prefer to omit the details.

Besides, note that the continuity in time is also ensured by the results of DiPerna \& Lions \cite[Section 2.II]{DiPernaLions} and the notion of renormalized solutions.

\end{proof}

%%%%%%%%%%%%%%%%%%%%%%%%%%%%%%%%%%%%%%%%%%%

%%%%%%%%%%%%%%%%%%%%%%%%%%%%%%%%%%%%%%%%%%%
\section{Some tools from optimal mass transportation}\label{sec:tools}

Let us remind the definition of the Wasserstein distance of order 1.
\begin{definition}\label{def:Wass}
Let $\mu$, $\nu$ be two Borel probability measures on $\reali^d$. We denote $\Xi\, (\mu,\nu)$  the set of \emph{plans}, that is the set of  probability measures $\gamma\in \Mes^+(\reali^d\times \reali^d)$ such that ${\mathbb P_x}_\sharp \gamma=\mu$ and ${\mathbb P_y}_\sharp \gamma=\nu$. We define the \emph{Wasserstein distance of order one} between $\mu$ and $\nu$  by
\begin{equation}\label{eq:wass1}
W_1(\mu,\nu)= \inf_{\gamma \in \,\Xi\, (\mu,\nu)} \int_{\reali^d\times\reali^d} \modulo{x-y}\d{\gamma (x,y)}\,.
\end{equation}

Let $\rho=(\rho^1, \ldots , \rho^k)$, $\sigma=(\sigma^1, \ldots , \sigma^k)$ be two  vectors such that $\rho^1, \ldots, \rho^k$ and $\sigma^1, \ldots, \sigma^k$ are  Borel probability measures on $\reali^d$. We define the \emph{Wasserstein distance of order one} between $\rho$ and $\sigma$, denoted $\mathscr{W}_1(\rho, \sigma)$, as
\begin{equation}\label{eq:wassk}
\mathscr{W}_1(\rho, \sigma)=\sum_{i=1}^k W_1(\rho^i, \sigma^i)\,.
\end{equation}
\end{definition}

\begin{remark}\label{rk:existmin}
By \cite[Theorem 1.3]{Villani}, for any $\mu, \nu\in \mathcal{P}(\reali^d)$, there exist a plan $\gamma_0\in \Xi(\mu,\nu)$ realizing the minimum in the Wasserstein distance so that
\[
W_1(\mu, \nu)=\int_{\reali^d} \modulo{x-y}\d{\gamma_0(x,y)}\,.
\]
\end{remark}

\begin{remark}\label{rk:inf}
Let $\bar \rho\in \Mes^+(\reali^d)^k$ be a probability measure ; and let $X, Y: \reali^d\to \reali^d$ be mappings such that $f =X_\sharp \d{\bar \rho}$ and $g = Y_\sharp \d{\bar \rho}$. Then, the probability measure $\gamma=(X, Y)_\sharp \d{\bar \rho}$ satisfies ${\mathbb{P}_x}_\sharp \gamma=f$, ${\mathbb{P}_y}_\sharp \gamma = g$ and so
\[
W_1(f,g)\leq \int_{\reali^d\times \reali^d }\modulo{x-y}\d{\gamma(x,y)}=\int_{\reali^d}\modulo{X-Y}\d{\bar\rho(x)}\,.
\]

\end{remark}

\begin{proposition}[cf. Villani {\cite[p. 207]{Villani}}] \label{prop:villani}
Let $\mu,\nu$ be two probability measures. The Wasserstein distance of order one between $\mu$ and $\nu$ satisfies
\[
W_1(\mu,\nu)=\sup_{\lip (\phi)\leq 1 } \int_{\reali^d}\phi(x)\left(\d{\mu(x)}-\d{\nu(x)}\right)\,.
\]
\end{proposition}

%%%%%%%%%%%%%%%%%%%%%%%%%%%%%%%%%%%%%%%%%%%

%%%%%%%%%%%%%%%%%%%%%%%%%%%%%%%%%%%%%%%%%%%
\section{The main stability estimate }\label{sec:stability}
%To prove Theorem \ref{thm:lagsol}, the idea is to fix the nonlocal term , to study the equation $\pt_t \rho+\div (\rho V(x, r))=0$ and to study the application $\mathscr{Q}:r\mapsto \rho$

In the following we consider probability measures instead of bounded positive measures. This is not a real restriction since we pass from one case to the other just by a rescaling.

Before giving a stability estimate in Proposition \ref{prop:stabGeneral}, we prove a technical lemma.

\begin{lemma}\label{lem:stability} 
Let $V$ satisfy \textbf{(V)} and $\eta $ satisfy \textbf{($\boldsymbol{\eta}$)}.  Let $r, s\in \mathcal{P}(\reali^d)^k$. 
For any $i\in\{1, \ldots, k\}$,  we have the following estimate
\begin{eqnarray*}
\norma{V^i(t, x,r\conv \eta_t^i)-V^i(t, x, s\conv \eta_t^i)}_{\L\infty}
&\leq & \lip_r( V^i)\, \lip_x( \eta^i)\, \mathscr{W}_1(r, s)\,.
\end{eqnarray*}
\end{lemma}
In the previous lemma, the quantity $\mathscr{W}_1(r, s)$ could be infinite. If we restrict ourselves to bounded positive measures with first moment finite, then the quantity above is always finite.

\begin{proof} The proof follows from Proposition \ref{prop:villani}  on the Wasserstein distance. 
Note first that in the case $\lip(\eta^{i,j})=0$ then $\eta^{i,j}$ is constant and we have $ (r^j-s^j)*\eta^{i,j}(x)=0\leq \lip(\eta^{i,j})W_1(r^j, s^j)$.
Now, in the case $\lip(\eta^{i,j})\neq 0$, thanks to Proposition \ref{prop:villani}, we have
\begin{align*}
(r^j-s^j)*\eta^{i,j}(x)&=\int_{\reali^d}\eta^{i,j}(x-y)(\d{r^j(y)}-\d{s^j(y)})\\
&= \lip(\eta^{i,j}) \int_{\reali^d} \frac{\eta^{i,j}(x-y)}{\lip(\eta^{i,j})}(\d{r^j(y)}-\d{s^j(y)})\\
&\leq \lip(\eta^{i,j}) \sup_{\lip(\phi)\leq 1} \int_{\reali^d}\phi(y) (\d{r^j(y)}-\d{s^j(y)})=  \lip(\eta^{i,j})W_1(r^j, s^j)\,.
\end{align*}
As we obtain the same estimate for $-(r^j-s^j)*\eta^{i,j}(x)$, we can conclude that
\begin{eqnarray*}
\norma{V^i(t, x,r\conv \eta^i)-V^i(t, x, s\conv \eta^i)}_{\L\infty}
&\leq & \lip_r (V^i) \norma{ (r-s) \conv \eta^i}_{\L\infty}\\
&\leq & \lip_r (V^i)\, \lip( \eta^i)\, \mathscr{W}_1(r,s)\,.
\end{eqnarray*}
\end{proof}

Let $r, s\in \L\infty([0,T], \mathcal{P}(\reali^d)^k)$.
We want to compare the following equations, in which the nonlocal has been replaced by fixed applications, so that the system is made of decoupled equations. 
\begin{equation}\label{eq:system2}
\begin{array}{l}
\textrm{for all } i \in \{1,\ldots, k\}\qquad \quad \pt_t \rho^i +\div \left(\rho^i V^i(t,  x, \eta^i \conv r)\right)=0\,,\qquad \quad\rho^i(0,\cdot)=\bar\rho^i \,,\\[5pt]
\textrm{for all } i \in \{1,\ldots, k\}\qquad  \quad \pt_t \sigma^i +\div \left(\sigma^i U^i(t,  x, \nu^i \conv s)\right)=0\,,\qquad \quad \sigma^i(0,\cdot)=\bar\sigma^i \,.
\end{array}
\end{equation}

\begin{proposition}\label{prop:stabGeneral} Assume $V, U$ satisfy \textbf{(V)} and $\eta, \nu $ satisfy \textbf{($\boldsymbol{\eta}$)}.
Let $\bar \rho$, $\bar \sigma$ be two probability measures such that for any $i$, $\norma{\bar\rho^i}_{\Mes}=\norma{\bar \sigma^i}_{\Mes}$. 
Let $r, s\in \L\infty([0,T],\mathcal{P}(\reali^d)^k)$. 
If $\rho$ and $\sigma$ are Lagrangian solutions of (\ref{eq:system2}) associated to the initial conditions $\bar \rho$ and $\bar \sigma$,  then we have the estimate:
\begin{eqnarray}
 \mathscr{W}_1 (\rho_t,\sigma_t) &\leq & e^{Ct} \mathscr{W}_1(\bar\rho, \bar\sigma) + Cte^{Ct}\left[  \sup_{t\in [0,T]} \mathscr{W}_1 (r_t,s_t) + \norma{\eta-\nu}_{\L\infty}+ \norma{V-U}_{\L\infty} \right]\label{eq:estimate}
\end{eqnarray}
where $C$ is a constant depending on $\lip_x(V)$, $\lip_r(V)$, $\lip_x(\eta)$ and $\norma{\bar\rho}_{\Mes}$.

Furthermore, in the special case $r=\rho$ and $s=\sigma$, we get:
\begin{eqnarray}
 \mathscr{W}_1(\rho_t, \sigma_t)&\leq& e^{Kt}\mathscr{W}_1(\bar\rho, \bar \sigma)+K te^{Kt}  \left[\norma{\eta-\nu}_{\L\infty} + \norma{V-U}_{\L\infty}\right]\,,\label{eq:stab}
\end{eqnarray}
where $K$is a constant depending on $\lip_x(V)$, $\lip_r(V)$, $\lip_x(\eta)$ and $\norma{\bar\rho}_{\Mes}$. 
\end{proposition}
Note that the estimate above comprises the case $\mathscr{W}_1(\bar\rho, \bar \sigma)=\infty$. 

%\begin{remark}
%In the same setting as Proposition \ref{prop:stabGeneral}, if the hypothesis $\norma{\bar\rho_i}_{\Mes} = \norma{\bar\sigma^i}_{\Mes}$ is not satisfied, we can compare $\frac{\rho^i}{\norma{\rho^i}_{\Mes}}$ and $\frac{\sigma^i}{\norma{\sigma^i}_{\Mes}}$ but the physical meaning of it is not clear. 
%\end{remark}

\begin{proof}
Let $\rho$, $\sigma$ be two Lagrangian solutions to  the Cauchy problem for (\ref{eq:system}) with initial conditions $\bar \rho$ and $\bar \sigma$ respectively. Let $X$, $Y$ be the associated ODE flows. For any $t\in [0,T]$,  we define the map $X^i_t \Join Y^i_t:\reali^d\times\reali^d\to \reali^d\times\reali^d$ by 
$$
X^i_t \Join Y^i_t(x,y)=(X^i_t(x), Y^i_t(y))\,,\qquad \qquad\textrm{ for any } (x,y)\in \reali^d\times\reali^d\,.
$$
Let $\gamma_0^i\in \Xi\, (\bar\rho^i, \bar\sigma^i)$ so that   ${\mathbb{P}_x}_\sharp \gamma^i_0=\bar\rho^i$ and ${\mathbb{P}_y}_\sharp \gamma^i_0=\bar\sigma^i$.  Let us define the probability measure $\gamma^i_t=(X^i_t\Join Y^i_t)_\sharp\gamma^i_0$. Then, ${\mathbb{P}_x}_\sharp\gamma^i_t = \rho^i_t$ and ${\mathbb{P}_y}_\sharp\gamma^i_t =\sigma^i_t$ so that $\gamma^i_t\in \Xi\, (\rho^i_t, \sigma^i_t)$.

Let $R>0$, we define, for $t\geq 0$  
\begin{align*}
Q_R(t)&=\sum_{i=1}^k\int_{X_t^i(B_R)\times Y_t^i(B_R)} \modulo{ x-y } \d{\gamma^i_t(x,y) } =\sum_{i=1}^k\int_{B_R\times B_R} \modulo{ X^i_t(x)-Y^i_t(y) } \d{\gamma^i_0 (x,y)} \,.
\end{align*}

Note first that $Q_R$ is Lipschitz. Indeed, let $t, s\geq 0$, then we have
\begin{eqnarray}
\modulo{Q_R(t)-Q_R(s)}
&\leq & \modulo{\sum_{i=1}^k \int_{B_R\times B_R}\left(\modulo{X^i_t(x)-Y^i_t(y) } -\modulo{X^i_s(x)-Y^i_s(y)} \right)\d{ \gamma^i_0 (x,y)} } \nonumber\\
&\leq & \sum_{i=1}^k \int_{B_R\times B_R}\modulo{X^i_t(x)-Y^i_t(y)- X^i_s(x)+ Y^i_s(y)} \d{\gamma^i_0 (x,y)}\nonumber\\
&\leq & \sum_{i=1}^k \int_{B_R\times B_R}\left(\modulo{X^i_t(x)-X^i_s(x)} +\modulo{Y^i_t(y)-Y^i_s(y)} \right) \d{\gamma^i_0 (x,y)}\nonumber \\
&\leq & \sum_{i=1}^k \int_{B_R\times B_R} \left(\norma{V^i}_{\L\infty} +\norma{U^i}_{\L\infty}\right) \modulo{t-s} \d{\gamma^i_0 (x,y)}\nonumber\\
&\leq & \left(\norma{V}_{\L\infty}+\norma{U}_{\L\infty}\right) \norma{\gamma_0}_{\Mes} \modulo{t-s} \,. \label{eq:Qlip}
\end{eqnarray}

Let us assume that $\mathscr{W}_1(\bar \rho, \bar \sigma)<\infty$, otherwise the thesis is trivial. Then, by Remark \ref{rk:existmin}, for all $i\in \{1,..., k\}$,  we can find a bounded positive measure $\gamma_0^i\in \Xi\, (\bar\rho^i, \bar\sigma^i)$ so that
$$
W_1(\bar \rho^i, \bar \sigma^i)= \int_{\reali^d \times \reali^d}\modulo{x-y}\d{\gamma_0^i(x,y)} \,.
$$
Consequently we have, for any $R\geq 0$, $Q_R(0)\leq \mathscr{W}_1(\bar \rho, \bar \sigma)$.
Hence, using (\ref{eq:Qlip}), for any $t\geq 0$, we have
\begin{eqnarray*}
Q_R(t)&\leq &Q_R(0)+ ( \norma{U}_{\L\infty} +\norma{V}_{\L\infty})\norma{\gamma_0}_{\Mes} t\\
&\leq & \mathscr{W}_1(\bar \rho, \bar \sigma) + (\norma{U}_{\L\infty}
+\norma{V}_{\L\infty}) \norma{\gamma_0}_{\Mes} t\,.
\end{eqnarray*}
Thus, for any $t\geq 0$,  $Q_R(t)$ remains finite when $R\to \infty$ and since $R\mapsto Q_R(t)$ is increasing with respect to $R$,  we can define $Q(t)=\lim_{R\to \infty} Q_R(t)$.

Let us now consider $Q$. 
The same computation as in (\ref{eq:Qlip}) ensures that $Q$ is Lipschitz so we can differentiate for almost every $t$ and obtain
\begin{eqnarray*}
Q'(t)
&\leq & \sum_{i=1}^k\int_{\reali^d\times \reali^d} \modulo{V^i\left(t,X^i_t(x), r_t\conv \eta_t^i(X^i_t(x))\right)-U^i\left(t,Y^i_t(y), s_t\conv \nu_t^i(Y^i_t(y))\right)}\d{\gamma^i_0 (x,y)}\\
&\leq & \sum_{i=1}^k\int_{\reali^d\times \reali^d} \bigg(  \modulo{V^i\left( t,X^i_t(x), r_t\conv \eta_t^i(X^i_t(x))\right)-V^i\left(t,Y^i_t(y), r_t\conv \eta_t^i(X^i_t(x))\right)} \\
&&\qquad+ \modulo{V^i\left(t,Y^i_t(y), r_t\conv \eta_t^i(X^i_t(x))\right)-V^i\left(t,Y^i_t(y), r_t\conv \eta_t^i(Y^i_t(y)) \right)}\\[2pt]
&&\qquad+ \modulo{V^i\left(t,Y^i_t(y), r_t\conv \eta_t^i(Y^i_t(y))\right)-V^i\left(t,Y^i_t(y), s_t \conv \eta_t^i(Y^i_t(y))\right)}\\[2pt]
&&\qquad+ \modulo{V^i\left(t,Y^i_t(y), s_t\conv \eta_t^i(Y^i_t(y))\right)-V^i\left(t,Y^i_t(y), s_t \conv \nu_t^i(Y^i_t(y))\right)}\\
&&\qquad+ \modulo{V^i\left(t,Y^i_t(y), s_t\conv \nu_t^i(Y^i_t(y))\right)-U^i\left(t,Y^i_t(y), s_t \conv \nu_t^i(Y^i_t(y))\right)}
\bigg)\d{\gamma^i_0 (x,y)}\,.
\end{eqnarray*}
Note that
\begin{eqnarray*}
\modulo{V^i(t,y,r_t*\eta_t^i(x))-V^i(t,y,r_t*\eta_t^i(y))}&\leq & \lip_r(V^i) \lip(r_t\conv \eta_t^i)\modulo{x-y}\\
&\leq & \lip_r(V^i) \norma{r_t}_{\Mes}\lip(\eta^i) \modulo{x-y}\,.
\end{eqnarray*}
%Furthermore, $\norma{r^i_t}_{\Mes}=\norma{s^i_t}_{\Mes}=\norma{\bar \rho^i}_{\Mes}$. Note also that $\norma{\gamma_0}_{\Mes}=\norma{\bar \rho}_{\Mes}$: indeed, as ${\mathbb{P}_x}_\sharp\gamma_0^i=\bar\rho^i$, we can write
%\begin{align}
%\norma{\bar\rho^i}_{\Mes}=\bar\rho^i(\reali^d)&={\mathbb{P}_x}_\sharp\gamma_0^i(\reali^d)=\gamma_0^i(\reali^d\times\reali^d)=\norma{\gamma_0^i}_{\Mes}\,.\label{eq:mes}
%\end{align} 
Using Lemma \ref{lem:stability} we obtain
\begin{eqnarray}
Q'(t)&\leq & \sum_{i=1}^k\int_{\reali^d\times \reali^d} \left(\lip_x(V^i)+\lip_r(V^i)\norma{r_t}_\Mes \lip(\eta^i)\right)\modulo{X^i_t(x)-Y^i_t(y)}\d{\gamma^i_0 (x,y)}\nonumber\\
&&+  \sum_{i=1}^k\int_{\reali^d\times \reali^d}\lip_r( V^i)\, \lip( \eta^i)\, \mathscr{W}_1(r_t, s_t) \d{\gamma^i_0 (x,y)}\nonumber \\
&&+ \sum_{i=1}^k\int_{\reali^d\times \reali^d} \lip_r( V^i) \modulo{s_t\conv (\eta_t^i-\nu_t^i)}\d{\gamma^i_0 (x,y)}+ \sum_{i=1}^k\int_{\reali^d\times \reali^d}  \norma{V^i-U^i}_{\L\infty}\d{\gamma^i_0 (x,y)}\nonumber \\
&\leq &  C\left[ Q(t)+ \mathscr{W}_1(r_t, s_t) +\norma{ \eta-\nu}_{\L\infty}+ \norma{U-V}_{\L\infty}\right]\,.\label{eq:stima}
\end{eqnarray}

Taking the $\sup$ in time of $\mathcal{W}_1(r_t, s_t)$ on the right-hand side and 
applying Gronwall Lemma, we get
\begin{eqnarray*}
Q(t)&\leq &e^{Ct}Q(0)+ C te^{Ct}\left(\sup_{\tau\in [0,t]}\mathscr{W}_1(r_\tau,s_\tau)+ \norma{ \eta-\nu}_{\L\infty}+ \norma{U-V}_{\L\infty}\right)\,.
\end{eqnarray*}
Note now that, thanks to remark (\ref{rk:inf}),  for any $t\geq 0$
\begin{eqnarray}
\mathscr{W}_1(\rho_t, \sigma_t)&\leq &Q(t)\,.\label{eq:WQ}
\end{eqnarray} 
Furthermore, we have chosen $\gamma_0$ in an optimal way thanks to Remark \ref{rk:existmin} so that $Q(0)=\mathscr{W}_1(\bar\rho, \bar \sigma)$. Hence we obtain, for any $t\in [0,T]$:
\begin{eqnarray*}
 \mathscr{W}_1(\rho_t, \sigma_t)&\leq &e^{Ct}\mathscr{W}_1(\bar\rho, \bar \sigma)+Ct e^{Ct} \left(\sup_{\tau\in [0,t]}\mathscr{W}_1(r_\tau,s_\tau)+ \norma{ \eta-\nu}_{\L\infty}+ \norma{U-V}_{\L\infty}\right)\,,
\end{eqnarray*}
which is the expected result (\ref{eq:estimate}).

In the particular case $r=\rho$ and $s=\sigma$, applying (\ref{eq:WQ}) to (\ref{eq:stima}) we obtain
\[
 Q'(t)\leq 2C Q(t)+C\left( \norma{ \eta-\nu}_{\L\infty}+ \norma{U-V}_{\L\infty}\right)\,.
\]
Applying Gronwall Lemma, we finally obtain $Q(t)\leq e^{2C t}Q(0)+Cte^{2Ct}\left(\norma{ \eta-\nu}_{\L\infty}+ \norma{U-V}_{\L\infty}\right)$, which is (\ref{eq:stab}).

\end{proof}

%%%%%%%%%%%%%%%%%%%%%%%%%%%%%%%%%%%%%%%%%%%

%%%%%%%%%%%%%%%%%%%%%%%%%%%%%%%%%%%%%%%%%%%
\section{Proof of the main theorem}\label{sec:proofmaintheorem}
The proof of Theorem \ref{thm:main} is based on the following  idea: let us fix the nonlocal term and, instead of (\ref{eq:system}), we study the Cauchy problem
\begin{equation}\label{eq:fix}
\pt_t \rho+\div (\rho\, V(t,x,r*\eta))=0\,, \qquad  \rho(0)=\bar \rho\,,
\end{equation}
where $r$ is a  given application. We consider here probability measures. In the more general case of positive measures with the same total mass, by rescaling we are back to the case of probability measures.

 Let us  introduce  the application
\begin{equation}\label{eq:Q}
\mathscr{Q}\;:\;\left\{
\begin{array}{ccc}
r & \mapsto & \rho\\
\mathscr{X}&\to &\mathscr{X}
\end{array}
\right\}\,,
\end{equation} 
where we consider the space $\mathscr{X}=\L\infty([0,T], \mathcal{P}(\reali^d)^k)$ for $T$ chosen in such a way that: 
\begin{description}
\item[(a)] The space $\mathscr{X}$ is equipped with a distance $d$ that makes $\mathscr{X}$ complete: for $\mu, \nu\in \mathscr{X}$, we define 
$$
d(\mu, \nu)=\sup_{t\in [0,T]} \mathscr{W}_1(\mu_t, \nu_t)\,.
$$
\item[(b)] The application $\mathscr{Q}$ is well-defined: the Lagrangian solution $\rho\in \mathscr{X}$ to (\ref{eq:fix}) exists and is unique (for  a fixed $r$). Indeed, let $X_t$ be the ODE flow associated to $ V(t, x, r_t\conv\eta_t)$, then we can define $\rho_t={X_t}_\sharp \bar \rho$. Since $\bar\rho$ is a positive measure, then so is $\rho_t$. 
\item[(c)] The application $\mathscr{Q}$ is a contraction: this is given by Proposition \ref{prop:stabGeneral}. 
Indeed, 
let $r, s$  in $ \L\infty([0,T], \Mes^+(\reali^d)^k)$ and denote $\rho=\mathscr{Q}(r)$, $\sigma=\mathscr{Q}(s)$ the associated solutions to (\ref{eq:fix}). Note that $\rho$ and $\sigma$ have the same initial condition. Thanks to Proposition \ref{prop:stabGeneral}, we obtain the contraction estimates
\[
 \sup_{[0,T]}\mathscr{W}_1(\rho_t, \sigma_t)\leq C T e^{CT} \sup_{[0,T]}\mathscr{W}_1(r_t, s_t)\,,
\]
where $C$ depends only on $\lip_x(V)$, $\lip_r(V)$, $\lip_x(\eta)$ and $\norma{\bar\rho}_{\Mes}$. 
\end{description}
Hence,  for $T$ small enough, by the Banach  fixed point Theorem we obtain existence and uniqueness in $\mathscr{X}$ of a Lagrangian solution to (\ref{eq:system}) for $t\in [0,T]$.
As $\norma{\rho_T}_{\Mes}=\norma{\bar\rho}_{\Mes}$ the coefficient $C$ does not depend on time and  we can iterate the procedure. Thus we have existence and uniqueness on $[0,+\infty[$.

Observe that uniqueness can be also obtained directly  by the stability estimate (\ref{eq:stab}) in the particular case   $V^i=U^i$, $\eta^i=\nu^i$, $\bar\rho=\bar\sigma$.

%%%%%%%%%%%%%%%%%%%%%%%%%%%%%%%%%%%%%%

%%%%%%%%%%%%%%%%%%%%%%%%%%%%%%%%%%%%%%%%%%%
%\section{Uniqueness and stability with respect to initial conditions}\label{sec:uniqueness}
%
%\begin{proposition}\label{prop:stability}
%Let $V \in (\L\infty\cap\lip)(\reali^d\times \reali^k, \M_{d,k}(\reali))$, $\eta \in (\L\infty\cap\lip )(\reali^d, \M_k(\reali))$. Let $\rho, \sigma$ be two Lagrangian solutions to (\ref{eq:system}) with initial conditions $\bar \rho, \bar \sigma \in \Mes^+(\reali^d)^k$ such that, for any $i\in \{ 1, ..., k\}$, $\norma{\bar\rho^i}_{\Mes}=\norma{\bar \sigma^i}_{\Mes}$. Then,
%for any $t\geq 0$ we have the stability estimate:
%\[
%\mathscr{W}_1(\rho_t, \sigma_t) \leq e^{Ct}\mathscr{W}_1(\bar\rho, \bar \sigma)\,,
%\]
%where $C$ is a constant depending on $\lip(V)$, $\lip(\eta)$ and $\norma{\bar \rho}_{\Mes}$.
%
%In particular, if $\bar \rho=\bar \sigma$ then we have $\rho_t=\sigma_t$ for any $t\geq 0$. In other words, if a Lagrangian solution to  the Cauchy problem for (\ref{eq:system}) exists, then it is unique.
%\end{proposition}

\small{

  \bibliography{nonlocal} 
}

\end{document}